\documentclass[12pt,reqno]{nyjm}
\usepackage{amsmath,bm}
\usepackage{amsfonts}
\usepackage{amssymb}
\usepackage{amsthm}
\usepackage{amscd}
\usepackage{a4wide}

\usepackage{amsmath}
\usepackage{amssymb}
\usepackage{amsthm}
\usepackage{graphicx}
\usepackage{eucal}
\usepackage{mathrsfs}
\usepackage{xy}
\xyoption{all}

\usepackage{color}
\usepackage[pdftex,bookmarks,colorlinks,breaklinks]{hyperref}  

\hypersetup{linkcolor=red,citecolor=blue,filecolor=dullmagenta,urlcolor=darkblue} 

\newtheorem {theorem}    {Theorem}[section]
\newtheorem {lemma}      [theorem]    {Lemma}

\theoremstyle{definition}

\newcounter{AbcT}

\numberwithin{equation}{section}

\newcommand{\A}{\mathcal A}
\newcommand{\B}{\mathcal B}

\newcommand {\R} {{\mathbb R}}

\newcommand {\Z} {{\mathbb Z}}

\newcommand{\IGNORE}[1]{}
\newcommand{\inv}{^{-1}}

\renewcommand{\limsup}{\varlimsup}

 \DeclareMathOperator{\SL}{SL}

\newcommand{\qp}{\mathbb{Q}_p}

\newcommand{\bbF}{\mathbb{F}}

\newcommand{\id}{\mathrm{id}}

\begin{document}
\title[Field Extensions]{Buildings, Extensions, and Volume Growth Entropy}

\begin{abstract}
Let $F$ be a non-Archimedean local field and let $E$ be a finite extension of $F$. Let $G$ be a split semisimple $F$ group. We discuss how to compare volumes on the Bruhat-Tits buildings $\B_E$ and $\B_F$ of $G(E)$ and $G(F)$ respectively.  

\end{abstract}

\author{J. S. Athreya}
\address{Department of Mathematics, University of Illinois.
1409 W. Green Street, Urbana, IL 61801.}
\author{Anish Ghosh}
\address{School of Mathematics, University of East Anglia, Norwich, NR4 7TJ UK}
\author{Amritanshu Prasad}
\address{The Institute of Mathematical Sciences, Taramani, Chennai 600 113, India}
\thanks{Ghosh is partly supported by a Royal Society Grant.}
    \thanks{J.S.A. partially supported by NSF grant DMS 1069153}

\maketitle
\tableofcontents

\section{Introduction}
Let $F$ be a non-Archimedean local field, so $F$ is a finite extension of the $p$-adic numbers $\qp$ or of $\bbF_{p}((t))$, the field of Laurent series over a finite field of $p$ elements. Let $E$ be an extension of $F$ of degree $n$. 
Let $G$ be a split semisimple linear algebraic group defined over $F$ and let $G(F)$ be the locally compact group of its $F$ points. The affine Bruhat-Tits building $\B_F$ of $G(F)$ is a simplicial complex on which $G(F)$ acts isometrically and plays a crucial role in understanding the structure and representations of $G(F)$. Let $\B_E$ denote the Bruhat-Tits building corresponding to $E$. Then $\B_F$ can be thought of as a sub-building of $\B_E$ and it is natural to compare properties of $\B_F$ and $\B_E$. For example take $G = \SL_2$, then $\B_F$ is a $q + 1$ regular tree where $q$ is the cardinality of the residue field of $F$. If $E$ is an unramified quadratic extension of $F$, then the bigger tree $\B_E$ is $q^2 + 1$ regular (see fig. \ref{fig:tree}). In fact, the question of how the tree $\B_F$ sits in the tree $\B_E$ turns out to depend on the ramification properties of the extension. We refer the reader to \S 5 in \cite{Kottwitz} for a lovely discussion.

\begin{figure}[htb!]
\centering%
\includegraphics[width=0.5\textwidth]{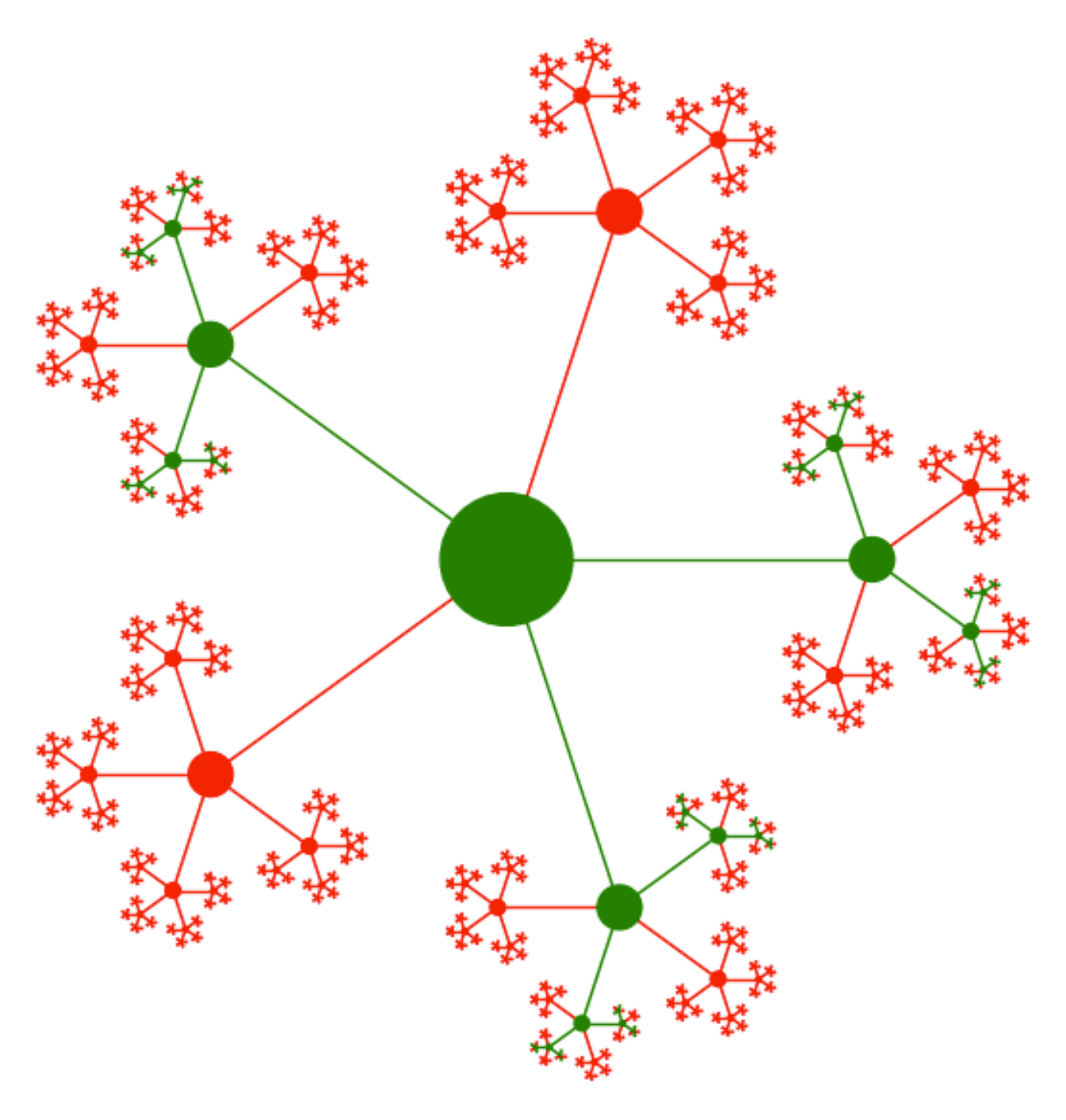}
\caption{ Embedded trees for $\SL_2$ (unramified case)}
\label{fig:tree}
\end{figure}

\vskip 0.2in

Now let $E/F$ be as above and let $f$ be the degree of the residue field extension. Then $e=n/f$ is the degree of ramification of $E$ over $F$. Fix valuations $v$ and $w$ on $F$ and $E$ and denote by $O_F$ (resp. $O_E$) the respective valuation rings and by $P_F$ (resp. $P_E$), the respective prime ideals in $O_F$ (resp. $O_E$). Then $P_F O_E = P_E^e$; moreover if $\varpi$ is a uniformizing element in $P_E$ and $\pi$ is a uniformizing element in $P_F$ then $\pi=\varpi^e \cdot u$ for some unit in $O_E$. Let $q$ be the cardinality of the residue field of $F$. Let $B(E), N(E)$ (resp. $B(F), N(F)$) denote the $(B, N)$-pairs of $G$ with respect to $F$ (resp. $E$). Let $\B_F$ (resp. $\B_E$) denote the respective buildings attached to the $(B, N)$-pairs and let $d_F$ (resp. $d_E$) denote the $G(F)$ (resp. $G(E)$) invariant metrics.

In this short note, we compare volumes of balls in $\B_E$ and $\B_F$. Fix a base point in the building $\B_F$ and let  $K_F$ (resp. $K_E$) denote the stabilizer in $G_F$ (resp. $G_E$) of this basepoint. Let $\mu_{F}$ (resp. $\mu_E$) denote Haar measures on $G(F)$ (resp. $G(E)$) normalized to give the stabilizers measure $1$. These induce measures on the respective Bruhat-Tits buildings which we also denote $\mu_{F}$ (resp. $\mu_E$). We recall that the \emph{volume growth entropy} $h(X, d, \mu)$ of a simply connected metric space $(X, d)$ with respect to a Borel measure $\mu$ is given by the exponential growth rate of the volume of balls, that is,  $$h(X, d, \mu):= \limsup_{R \rightarrow \infty} \frac{\log \mu( B(x, R) )}{R}.$$ Note that $h$ is independent of the normalization of the measure $\mu$, that is, for any constant $c>0$, $$h(X, d, c\mu) = h(X, d, \mu).$$
With the inclusion $\B_F \hookrightarrow \B_E$, we have 3 metric measure spaces we would like to compare: $(\B_F, d_F, \mu_F)$, $(\B_F, d_E, \mu_F)$ and $(\B_E, d_E, \mu_E)$. Our main result regarding entropy is:

\begin{theorem}\label{theorem-entropy}The volume growth entropies are related by
$$h(\B_F, d_E, \mu_F) = \frac{1}{e} h(\B_F, d_F, \mu_F) = \frac{1}{n} h(\B_E, d_E, \mu_E),$$ or, equivalently
$$n h(\B_F, d_E, \mu_F) = f h(\B_F, d_F, \mu_F) =  h(\B_E, d_E, \mu_E).$$ 
\end{theorem}
\medskip
\noindent We discuss the construction of the Bruhat-Tits building in \S \ref{section:defn-building}. The proof of Theorem \ref{theorem-entropy} proceeds by direct computation of volumes of balls. Along the way, we compare metrics on $\B_F$ and $\B_E$ (\ref{eq:1}) a result which may be of independent interest. In \S \ref{sec:proof-trees}, we prove the theorem in the simplified case of trees before proving it in full generality in \S \ref{sec:proofs-main-theorems}.  The volume growth entropy for compact quotients of Bruhat-Tits buildings has been computed explicitly by Leuzinger~\cite{Leuzinger} who also showed that (appropriately normalized) that it is equal to the entropy of the geodesic flow. The volume comparison result in this paper has implications for homogeneous dynamics. Geodesic flows on quotients of  symmetric spaces and buildings have been  extensively studied. In \cite{AGP}, we proved an analogue of the logarithm laws of Sullivan (\cite{Sullivan}) and Kleinbock-Margulis (\cite{KleinMarg}) for function fields. These results describe the asymptotic behaviour of geodesic trajectories to shrinking cuspidal neighborhoods. Using Theorem \ref{theorem-entropy}, ``relative" versions of logarithm laws can be established. Details will appear elsewhere. 

\section*{Acknowledgements} Parts of this work was done when AG and JSA were visiting the Institute for Mathematical Sciences, Chennai. They thank the institute for its hospitality and excellent working conditions. 

\section{Definition of the building}
\label{section:defn-building}

Recall from \cite{BT1, BT2, Rabinoff} the construction of a principal apartment for $G$.
Let $T$ be a maximal split torus in $G$.
Denote by $X^*$ the real vector space $X^*(T)\otimes \R$.
Here, as is usual, $X^*(T)$ denotes the lattice of algebraic homomorphisms $T\to \mathbf G_m$.
The dual space $X_*$ can be identified with $X_*(T)\otimes \R$, where $X_*(T)$ is the lattice of cocharacters $\mathbf G_m\to T$.
Let $\Phi=\Phi(G,T)\subset X^*(T)$ denote the root system of $G$ with respect to $T$.

The affine apartment $\A(G,T)$ is just $X_*$, together with a hyperplane configuration $H_{\alpha+n}$, where
\begin{equation*}
	H_{\alpha+n}=\{x\in X_*|\alpha(x)+n=0\}
\end{equation*}
is an affine hyperplane in $X_*$ for each $\alpha\in\Phi$ and $n\in \Z$.

This hyperplane configuration allows us to think of $\A$ as  a simplicial complex.
The vertices of this simplicial complex are the points in the weight lattice
\begin{equation*}
	Q=\{x\in \A|\alpha(x)\in\Z\text{ for all } \alpha\in\Phi\}.
\end{equation*}

The affine linear functional $x\mapsto \alpha(x)+n$ is usually denoted by $\alpha+n$.
The affine root system is the set of affine linear functionals on $\A$ given by
\begin{equation*}
	\Psi=\{\alpha+n|\alpha\in\Phi,\;n\in\Z\}.
\end{equation*}
If $N$ denotes the normalizer $N_{G(F)}T$ of $T$ in $G(F)$, then $N$ contains $T(F)$ as a normal subgroup with quotient $W$, which is, by definition, the Weyl group of $G$ with respect to $T$.
In fact $N$ is a semidirect product:
\begin{equation*}
	N=T(F)\rtimes W
\end{equation*}

Fix a uniformizing element $\pi$ in the ring of integers of $F$.
Recall that, if for each $\mu\in X_*(T)$, we define $\pi^\mu\in T(F)$ to be the element $\mu(\pi)$, then $\mu\mapsto\pi^\mu$ gives rise to an isomorphism
\begin{equation*}
	X_*(T)\widetilde\to\frac{T(F)}{T(O)}.
\end{equation*}
Denote by the $\vartheta$ the inverse of the above isomorphism composed with the quotient map $T(F)\to T(F)/T(O)$.

The affine apartment $\A$ is an $N$-space (in the sense that there is an action of $N$ on $\A$ which preserves the hyperplane configuration).
The action is as follows:
\begin{equation*}
	(tw)\cdot \mu=\vartheta(t)+{}^w\mu
\end{equation*}
The reason that the hyperplane configuration is preserved is that 
\begin{align*}
	\alpha+n[(tw)\cdot\mu]&=\alpha(\vartheta(t))+{}^w\mu)+n\\
	&=\alpha^w(\mu)+\alpha(\vartheta(t))+n\\
	&=[\alpha^w+\alpha(\vartheta(t))](\mu)
\end{align*}
so that composing with the $N$-action on $\A$ takes an affine root to another affine root.

The action of $N$ on $\A$ factors through the quotient of $N$ by $T(O)$, which is called the affine Weyl group of $G$:
\begin{equation*}
  \widetilde W = \frac{T(F)}{T(O)}\rtimes W.
\end{equation*}

To each point $x\in\A$ is associated a parahoric subgroup $G_x$ of $G(F)$ such that $N\cap G_x$ is the isotropy subgroup of $x$ in $N$.
The Bruhat-Tits building of $G$ is constructed as follows:
\begin{equation*}
	\B:=(G(F)\times\A)/\sim
\end{equation*}
where \lq\lq$\sim$\rq\rq{} is the equivalence relation on acts on $G\times\A$ for which $(g,x)\sim(h,y)$ if there exits $n\in N$ such that
\begin{equation*}
	n\cdot x=y \text{ and } g\inv hn\in G_x
\end{equation*}
where $G_x$ is the parahoric subgroup corresponding to the point $x\in\A$.

For example, if $\id_G$ denotes the identity element of $G(F)$, then
\begin{equation*}
	(\id,x)\sim (\id,y)
\end{equation*}
if and only if there exists $n\in N\cap G_x$ such that $n\cdot x=y$, which amounts to requiring that $x=y$.
Thus, $\A$ is itself embedded in $\B$ as $\{\id_G\}\times \A$.

$G$ acts on $\B$ via
\begin{equation*}
	g\cdot(h,x)=(gh,x).
\end{equation*}

For example, under this action, $g\cdot(\id,x)\sim (\id,x)$ if and only if $(g,x)\sim(\id,x)$, or in other words, if and only if there exists $n\in N$ such that $n\cdot x=x$ and $gn\in G_x$.
Since $N\cap G_x$ is the isotropy subgroup of $x$ in $N$, $n$ itself must lie in $G_x$.
Therefore $G_x$ is the isotropy subgroup of $(\id,x)$ in $G$.
More generally, $g\cdot(h,x)\sim(h,x)$ if and only if $(gh,x)\sim(h,x)$ if and only if $h\inv gh\in G_x$.
In other words, the stabilizer of $(h,x)$ is the parahoric subgroup $hG_xh\inv$.

Had we chosen a different split torus $T'$ which was conjugate to $T$, we would have begun with an apartment $\A'$ corresponding to $T'$.
We would always be able to find $g\in G(F)$ such that $gTg\inv=T'$
The building $\B'$ constructed from $\B$ would be isomorphic to $\B$ as a $G(F)$-space by identifying $\A'$ with $g\cdot \A\subset\B$.

These subsets $g\cdot A$ are known as the apartments of $\B$.
A basic fact about the building is that any two points are contained in an apartment.

Thus, in order to define a metric on $\B$, we take the following strategy: given $x,y\in \B$, we find an apartment $g\cdot\A$ such that $x$ and $y$ lie in $g\cdot \A$.
In other words, $g\inv x$ and $g\inv y$ lie in $\A$.
Now $\A$ itself has, up to scaling, a unique $W$-invariant inner product, and it is the distance between $g\inv x$ and $g\inv y$ with respect to a fixed such inner product that we declare to be the distance between $x$ and $y$.

We may normalize the metric on $\B$ by normalising the $W$-invariant inner product on $\A$, which may be achieved by declaring that the diameter of each connected component of the complement of the union of the hyperplane configuration $H_{\alpha+n}$ as $\alpha+n$ varies over the set $\Psi$ of affine roots has diameter one.

The building $\B$ inherits the structure of a simplicial complex from the apartment $\A$.
A simplex in $\B$ is a $G$-translate of a simplex in $\A$.
\section{Behaviour under field extensions}

\begin{equation*}
\begin{xy}
	(0,0)*{\bullet};
       (5,5)*{\circ}**@{-};
       (0,0)*{\bullet};
       (0,10)*{\circ}**@{-};
       (5,5)*{\circ}**@{.};
       (10,10)*{\bullet}**@{-};
       (5,15)*{\circ}**@{-};
       (0,10)*{\circ}**@{.};
       (5,15)*{\circ};
       (5,5)*{\circ}**@{.};
       (0,10)*{\circ};
       (0,20)*{\bullet}**@{-};
       (5,15)*{\circ}**@{-};
       (10,20)*{\circ}**@{.};
       (10,10)*{\bullet}**@{-};
       (15,5)*{\circ}**@{-};
       (5,5)*{\circ}**@{.};
       (0,10)*{\circ}**@{.};
       (10,0)*{\circ}**@{.};
       (0,0)*{\bullet}**@{-};
       (10,0)*{\circ};
       (15,5)*{\circ}**@{.};
       (20,0)*{\bullet}**@{-};
       (10,0)*{\circ}**@{-};
       (20,0)*{\bullet};
       (20,10)*{\circ}**@{-};
       (15,5)*{\circ}**@{.};
       (15,15)*{\circ}**@{.};
       (10,10)*{\bullet}**@{-};
       (15,15)*{\circ};
       (20,10)*{\circ}**@{.};
       (15,15)*{\circ};
       (10,20)*{\circ}**@{.};
       (15,25)*{\circ}**@{.};
       (20,20)*{\bullet}**@{-};
       (20,10)*{\circ}**@{-};
       (15,15)*{\circ};
       (20,20)*{\bullet}**@{-};
       (15,15)*{\circ};
       (15,25)*{\circ}**@{.};
       (10,20)*{\circ}**@{.};
       (10,30)*{\bullet}**@{-};
       (10,20)*{\circ};
       (5,25)*{\circ}**@{.};
       (5,35)*{\circ}**@{.};
       (0,30)*{\circ}**@{.};
       (0,20)*{\bullet}**@{-};
       (5,25)*{\circ}**@{-};
       (0,30)*{\circ}**@{.};
       (5,25)*{\circ};
       (10,30)*{\bullet}**@{-};
       (5,35)*{\circ}**@{-};
       (5,25)*{\circ};
       (5,15)*{\circ}**@{.};
       (0,20)*{\bullet};
       (0,30)*{\circ};
       (0,40)*{\bullet}**@{-};
       (5,35)*{\circ}**@{-};
       (10,30)*{\bullet};
       (15,25)*{\circ}**@{-};
       (20,30)*{\circ}**@{.};
       (20,20)*{\bullet}**@{-};
       (20,30)*{\circ};
       (20,40)*{\bullet}**@{-};
       (15,35)*{\circ}**@{-};
       (15,25)*{\circ}**@{.};
       (20,30)*{\circ};
       (15,35)*{\circ}**@{.};
       (10,30)*{\bullet}**@{-};
       (5,35)*{\circ};
       (15,35)*{\circ}**@{.};
       (10,40)*{\circ};
       (5,35)*{\circ}**@{.};
       (0,40)*{\circ};
       (10,40)*{\circ}**@{-};
       (20,40)*{\bullet}**@{-};
       (15,35)*{\circ};
       (10,40)*{\circ}**@{.};
\end{xy}
\end{equation*}
\begin{center}
Embedded apartments for type $C_2$
\end{center}

\vskip 0.2in

Let $E$ be a finite extension of $F$ of degree $n$.
Suppose that the residue field extension is of degree $f$.
Then $e=n/f$ is the degree of ramification of $E$ over $F$.
Let $P_F$ denote the prime ideal in the ring of integers of $F$.
Then $P_FO_E=P_E^e$.
If $\varpi$ is a uniformizing element in $P_E$ and $\pi$ is a uniformizing element in $P_F$ then $\pi=\varpi^e \cdot u$ for some unit in $O_E$.

Let $G$ be a split semisimple group over $F$ and fix a maximal $F$-split torus $T$.
Let $\A_E$ and $\A_F$ denote the apartments of $G(E)$ and $G(F)$ with respect to $T$.
As sets, these are exactly the same; they are both isomorphic to $X_*(T)\otimes \R$.
However, the identity map is not the correct way to identify $\A_F$ with $\A_E$.
Recall that $\A_F$ and $\A_E$ comes with actions of $N_F:=N_{G(F)}T$ and $N_E:=N_{G(E)}T$ respectively.
Let $i:N_F\hookrightarrow N_E$ denote the emedding of $N_F$ as a subgroup of $N_E$.
Let $e:\A_F\to A_E$ denote multiplication by $e$.
Then the following diagram commutes.
\begin{equation}
	\label{eq:actions}
	\xymatrix{
		N_F\times \A_F \ar[r] \ar[d]_{i\times e}\ar@{}[dr]|\circlearrowright& \A_F \ar[d]^e\\
		N_E\times \A_E \ar[r] & \A_E
	}
\end{equation}
Furthermore, 
\begin{equation}
	\label{eq:parahorics}
	G(E)_{e(x)}\cap G(F)=G(F)_x
\end{equation}
for each $x\in \A_F$.
We may apply the construction of Section~\ref{section:defn-building} to $G(F)$ as well as $G(E)$ resulting in two different buildings, which we denote by $\B_F$ and $\B_E$ respectively.
The equations (\ref{eq:actions}) and (\ref{eq:parahorics}) imply that the map $i\times e:G(F)\times A_F\to G(E)\times A_E$ descends to an inclusion of $\B_F$ in $\B_E$.

The stabilizer of the image of $(\id_{G(F)},0)$ in $\B_F$ (which is a vertex of $\B_F$) is $G(O_F)$, and similarly, the stabilizer of $(\id_{G(E)},0)$ in $\B_E$ is $G(O_E)$.

Since multiplication by $e$ is an isomorphism $\A_F\to \A_E$ of simplicial complexes, it follows that, for $x,y\in A_F$ (which can also be thought of as points of $\A_E$)
\begin{equation}
  \label{eq:1}
  d_E(x,y)=ed_F(x,y).
\end{equation}
Since any two points in $B_F$ are contained in some apartment, \eqref{eq:1} remains valid for any $x,y\in \B_F$.

\section{Trees}\label{sec:proof-trees}

\begin{equation*}
\begin{xy}
        (0,0)*{\bullet};
        (5,5)*{\circ}**@{-};
        (10,0)*{\circ}**@{.};
        (5,5)*{\circ};
        (10,10)*{\bullet}**@{-};
        (5,15)*{\circ}**@{-};
        (0,20)*{\bullet}**@{-};
        (5,15)*{\circ};
        (0,10)*{\circ}**@{.};
        (10,10)*{\bullet};
        (15,10)*{\circ}**@{-};
        (15,5)*{\circ}**@{.};
        (15,10)*{\circ};
        (20,10)*{\bullet}**@{-};
        (25,15)*{\circ}**@{-};
        (30,20)*{\bullet}**@{-};
        (25,15)*{\circ};
        (20,20)*{\circ}**@{.};
        (20,10)*{\bullet};
        (25,5)*{\circ}**@{-};
        (30,0)*{\bullet}**@{-};
        (25,5)*{\circ};
        (30,10)*{\circ}**@{.};      
\end{xy}
\end{equation*}

\begin{center}
Embedded trees for $\SL_2$ (totally ramified case)
\end{center}

\vskip 0.2in

 In this section we sketch a proof of Theorem~\ref{theorem-entropy}, in the case $G= \SL_2$, and thus where the buildings $\B_F$ and $\B_E$ are regular trees. We first decompose the field extension $E \rightarrow F$ as a chain of two extensions $E' \rightarrow F$ and $E \rightarrow E'$, where $E' \rightarrow F$ is a totally ramified extension of degree $e$ and $E \rightarrow E'$ is an unramified extension of degree $f$. We fix notation: $\log$ will denote natural logarithm, and let $\alpha_R(t)$ denote the volume of the ball (with respect to counting measure on vertices) of radius $R$ in a $(t+1)$-regular tree, that is $$\alpha_R(t) = (t+1)t^{R-1}.$$ We note that $$\limsup_{r \rightarrow \infty} \frac{\log \alpha_R(t) }{R} = \log t.$$ There is a prime power $q =p^n$ so that $\B_F$ is a $q+1$-regular tree, which immediately yields $$h(\B_F, d_F) = \log q.$$ Here, and below, we drop the measure $\mu$ from the definition of entropy as we will work exclusively with counting measure, taking advantage of the fact that the entropy $h$ is invariant under scaling. We treat the unramified and fully ramified cases in turn, before working out the general case. 

\subsection{Totally Ramified Extensions} As above, $E' \rightarrow F$ is a totally ramified extension of degree $e$. Then the tree $\B_{E'}$ is formed by subdividing each edge of $\B_F$ into $e$ segments by adding $e-1$ vertices. To each of these, we attach a rooted tree, where the root has valence $q-2$, and all the descendants have valence $q+1$ ($q$ descendants and one parent). Thus, the resulting $\B_{E'}$ is a $q+1$-regular tree. The metric $d_{E'}$ gives each edge length $1$, thus, we have $$\mu(B(R, E') \cap \B_F) = \mu\left(B\left(\frac{R}{e}, F\right)\right) = \alpha_{\frac{R}{e}} (q),$$ where $B(R, E')$ is the ball of radius $R$ in the metric $d_{E'}$ (we will also write $B(R, F)$ and $B(R, E)$ to mean balls of radius $R$ in the $d_F$ metric and $d_E$ metrics respectively). In this case, we have that the metric spaces $(\B_F, d_F)$ and $(\B_{E'}, d_{E'})$ are isometric (both being metric $q+1$-regular trees), but $(\B_F, d_{E'})$ is not isometric to either, since the inclusion $(\B_F, d_F) \hookrightarrow (\B_{E'}, d_{E'})$ is not an isometry.

\subsection{Unramified Extensions} Now consider the unramified extension $E \rightarrow E'$ of degree $f$. Here, the tree $\B_{E}$ is formed from the tree $\B_{E'}$ by adding $q^f-q$ edges to each vertex, and then rooted $q^f+1$-valent trees to each new vertex (that is, the root has $q^f$-children, and each descendant has $q^f$ further descendants). Here, the inclusion $\B_{E'}, d_{E'}) \hookrightarrow (\B_{E}, d_{E})$ is an isometry, but the metric spaces  $(\B_E, d_E)$ and $(\B_{E'}, d_{E'})$ are clearly not isometric. Thus, $$\mu(B(R, E) \cap \B_{E'}) = \mu(B(R, E')) = \alpha_R(q).$$

\subsection{The general case} To construct $\B_E$ from $\B_F$, we combine the two procedures, and thus obtain $$\mu(B(R, E) \cap \B_{F}) = \mu\left(B\left(\frac R e, F\right)\right) = \alpha_{\frac R e}(q).$$ Taking $\log_q$, dividing by $R$, and letting $R \rightarrow \infty$, we obtain, as desired $$h(\B_F, d_E, \mu_F) = \frac{1}{e} h(\B_F, d_F, \mu_F) = \frac{1}{ef} h(\B_E, d_E, \mu_E),$$ and noting that $ef=n$, we obtain Theorem~\ref{theorem-entropy}.

\section{Proof of the Main Theorem}
\label{sec:proofs-main-theorems}

The proof of the Main Theorem follows along similar lines to the argument in \S \ref{sec:proof-trees}. In the three lemmas below, we explicitly compute the volumes of balls in Bruhat-Tits buildings. The result follows upon taking ratios of volume entropies.

\begin{lemma}
  \label{lemma:distance_is_length}
  For any $x\in B$ and $w\in \widetilde W$,
  \begin{equation*}
    d(wx,x)\asymp l(w),
  \end{equation*}
\noindent where $l$ is the length function of the affine Weyl group $\widetilde W$.   
\end{lemma}
\begin{proof}
  For $\mu\in X_*(T)$ let $\varpi^\mu$ denote the image of $\mu(\varpi)$ in $T(F)/T(O)$, which we may regard as an element of $\widetilde W$.
  By \cite{MR2250034}, $l(\varpi^\mu)=\langle \rho,\mu\rangle$, where $\rho$ denotes half the sum of positive roots.
  Consider the hyperplane 
  \begin{equation*}
    H=\{x\in X_*\;|\;\langle\rho,x\rangle = 1\}.
  \end{equation*}
  If $x\in X_*$ is dominant and non-zero, then clearly, $\langle \rho,x\rangle >0$ (since $\langle  \alpha,x\rangle \geq 0$ for all $\alpha>0$).
  Therefore, this hyperplane intersects each ray in the dominant cone at exactly one point.
  It follows (since the dominant cone is pointed - see Theorem~1.26 of \cite{Mahatab}) that the intersection of $H$ with the dominant cone is compact.
  It follows that there exist positive constants $C$ and $c$ such that for any dominant point $x$ in $H$,
  \begin{equation*}
    c<\|x\|<C.
  \end{equation*}
  By scaling any dominant point of $X_*$ into $H$, it follows that
  \begin{equation*}
    c\langle\rho,x\rangle<\|x\|<C\langle \rho,x\rangle.
  \end{equation*}
  Now any $w\in \widetilde W$ can be written as $w=\varpi^\mu w_0$ for $w_0\in W$ and $\mu\in X_*(T)^{++}$.
  Since the lengths of elements in the finite Weyl group form a bounded set, and the metric in $\A$ is $W$-invariant, it follows that $d(w\cdot 0,0)\asymp l(w)$.
  Since $\B$ is homogeneous, we get the result for all $x\in \B$.
\end{proof}
 Let
\begin{equation*}
  S(q,R) : = \sum_{w\in \tilde W,\; l(w)\leq R} q^{l(w)}.
\end{equation*}

\begin{lemma}
  For every $R\geq 0$, 
  \begin{equation*}
    \mu(B(x,R))\asymp S(q, R).
  \end{equation*}
\end{lemma}
\begin{proof}
  By Lemma~\ref{lemma:distance_is_length}, $d(wx,x)\asymp l(w)$.
  By the $G(F)$-invariance of the metric on $\B$, it follows that $d(IwI,I)\asymp l(w)$ where $I$ is the Iwahori subgroup of $G(F)$.
\end{proof}
\begin{lemma}
  \label{lemma:volume-asymp}
  As a function of $R$,
  \begin{equation*}
    S(q, R)\asymp W(q)\prod_i(q^{m_i}-1)q^{Rr},
  \end{equation*}
  where $r$ is the (semisimple) rank of $G$, $W(t)$ is the Poincar\'e polynomial of the finite Weyl group of $G$, the product is over an indexing set of simple roots of $G$, and $m_i$ is the exponent of the corresponding simple root.
\end{lemma}
\begin{proof}
  The polynomial (in the variable $q$) $S(q, R)$ is the truncation of the Poincare series of $\widetilde W$ at $q^R$.
  This series is given by
  \begin{equation*}
    \widetilde W(q) = W(q)\prod_i \frac 1{1-q^{m_i}}
  \end{equation*}
  (see \cite[\S 8.9]{Hum}).
\end{proof}

\end{document}